\documentclass{article}%
\usepackage{amsmath}
\usepackage{amsfonts}
\usepackage{amssymb}
\usepackage{graphicx}%
\hoffset= - 0.5 cm 

\usepackage[usenames]{color}

\def\R{\mathbb R}

\newtheorem{theorem}{Theorem}[section]

\newtheorem{lemma}[theorem]{Lemma}

\newtheorem{proposition}[theorem]{Proposition}
\newtheorem{remark}[theorem]{Remark}

\newenvironment{proof}[1][Proof]{\noindent\textbf{#1.} }
{\ \rule{0.5em}{0.5em}}
\numberwithin{equation}{section}
\begin{document}

\title{Global $L^{p}$ estimates for degenerate Ornstein-Uhlenbeck operators with
variable coefficients\thanks{\textbf{Keywords:} Ornstein-Uhlenbeck operators.
Global $L^{p}$ estimates. Hypoelliptic operators. Singular integrals.
Nondoubling spaces. \textbf{Mathematics subject classification (2000)}:
Primary 35H10. Secondary 35B45 - 35K70 - 42B20.}}
\author{Marco Bramanti\\{\normalsize Dip. di Matematica, Politecnico di Milano}\\{\normalsize Via Bonardi 9, 20133 Milano, Italy}\\{\normalsize marco.bramanti@polimi.it}
\and Giovanni Cupini, \ \ Ermanno Lanconelli\\{\normalsize Dip. di Matematica, Universit\`{a} di Bologna}\\{\normalsize Piazza di Porta S. Donato 5, 40126 Bologna, Italy}\\{\normalsize giovanni.cupini@unibo.it, ermanno.lanconelli@unibo.it}
\and Enrico Priola\\{\normalsize Dip. di Matematica, Universit\`{a} di Torino}\\{\normalsize via Carlo Alberto 10, 10123 Torino, Italy}\\{\normalsize enrico.priola@unito.it}}
\maketitle

\begin{abstract}
We consider a class of degenerate  Ornstein-Uhlenbeck operators in
$\mathbb{R}^{N}$, of the kind%
\[
\mathcal{A}\equiv\sum_{i,j=1}^{p_{0}}a_{ij}\left(  x\right)  \partial
_{x_{i}x_{j}}^{2}+\sum_{i,j=1}^{N}b_{ij}x_{i}\partial_{x_{j}}%
\]
where $\left(  a_{ij}\right)  $ is symmetric uniformly positive
definite on $\mathbb{R}^{p_{0}}$ ($p_{0}\leq N$), with uniformly
continuous  and bounded entries, and $\left(  b_{ij}\right)  $
is a constant matrix such that the frozen operator
$\mathcal{A}_{x_{0}}$ corresponding to $a_{ij}\left(  x_{0}\right) $
is hypoelliptic. For this class of operators we prove global $L^{p}$
estimates
($1<p<\infty$) of the kind:%
\[
\left\Vert \partial_{x_{i}x_{j}}^{2}u\right\Vert _{L^{p}\left(  \mathbb{R}%
^{N}\right)  }\leq c\left\{  \left\Vert \mathcal{A}u\right\Vert _{L^{p}\left(
\mathbb{R}^{N}\right)  }+\left\Vert u\right\Vert _{L^{p}\left(  \mathbb{R}%
^{N}\right)  }\right\}  \text{ for }i,j=1,2,...,p_{0}.
\]
We obtain the previous estimates as a byproduct of the following one, which is
of interest in its own:%
\[
\left\Vert \partial_{x_{i}x_{j}}^{2}u\right\Vert _{L^{p}\left(  S_{T}\right)
}\leq c\left\{  \left\Vert Lu\right\Vert _{L^{p}\left(  S_{T}\right)
}+\left\Vert u\right\Vert _{L^{p}\left(  S_{T}\right)  }\right\}
\]
for any $u\in C_{0}^{\infty}\left(  S_{T}\right)  ,$ where $S_{T}$ is the
strip $\mathbb{R}^{N}\times\left[  -T,T\right]  $, $T$ small, and $L$ is the
Kolmogorov-Fokker-Planck operator%
\[
L\equiv\sum_{i,j=1}^{p_{0}}a_{ij}\left(  x,t\right)  \partial_{x_{i}x_{j}}%
^{2}+\sum_{i,j=1}^{N}b_{ij}x_{i}\partial_{x_{j}}-\partial_{t}%
\]
with uniformly continuous and bounded $a_{ij}$'s.
\end{abstract}

\section{Introduction}

Let us consider the following kind of Ornstein-Uhlenbeck operators:%
\begin{equation} \label{ou1}
\mathcal{A}=\sum_{i,j=1}^{p_{0}}a_{ij}\left(  x\right)
 \partial_{x_{i}x_{j}%
}^{2}+\sum_{i,j=1}^{N}b_{ij}x_{i}\partial_{x_{j}},
\end{equation}
where:

$A_{0}=\left(  a_{ij}\left(  x\right)  \right)  _{i,j=1}^{p_{0}}$ is a
$p_{0}\times p_{0}$ ($p_{0}\leq N$) symmetric, bounded and uniformly positive
definite matrix:%
\begin{equation}
\frac{1}{\Lambda}\left\vert \xi\right\vert ^{2}\leq\sum_{i,j=1}^{p_{0}}%
a_{ij}\left(  x\right)  \xi_{i}\xi_{j}\leq\Lambda\left\vert \xi\right\vert
^{2} \label{ellitpicity}%
\end{equation}
for all $\xi\in\mathbb{R}^{p_{0}},$ $x\in\mathbb{R}^{N}$ and for some constant
$\Lambda\geq1$;

the entries $a_{ij}$ are supposed to be uniformly continuous functions on
$\mathbb{R}^{N}$, with a modulus of continuity
\begin{equation}
\omega\left(  r\right)  =\max_{i,j=1,...,p_{0}}\sup_{\substack{x,y\in
\mathbb{R}^{N}\\\left\vert x-y\right\vert \leq r}}\left\vert a_{ij}\left(
x\right)  -a_{ij}\left(  y\right)  \right\vert ;\label{w}%
\end{equation}

the constant matrix $B=\left(  b_{ij}\right)  _{i,j=1}^{N}$ has the following
structure:%
\begin{equation}
B=%
\begin{bmatrix}
\ast & B_{1} & 0 & \ldots & 0\\
\ast & \ast & B_{2} & \ldots & 0\\
\vdots & \vdots & \vdots & \ddots & \vdots\\
\ast & \ast & \ast & \ldots & B_{r}\\
\ast & \ast & \ast & \ldots & \ast
\end{bmatrix}
\label{B}%
\end{equation}
where $B_{j}$ is a $p_{j-1}\times p_{j}$ block with rank $p_{j},j=1,2,...,r$,
$p_{0}\geq p_{1}\geq...\geq p_{r}\geq1$, $p_{0}+p_{1}+...+p_{r}=N$ and the
symbols $\ast$ denote completely arbitrary blocks.

If the $a_{ij}$'s are constant, the above assumptions imply that the operator
$\mathcal{A}$ is hypoelliptic (although degenerate, as soon as $p_{0}<N$), see
\cite{LP}. If the $a_{ij}$'s are just uniformly continuous, $\mathcal{A}$ is a
nonvariational degenerate elliptic operator with continuous coefficients,
structured on a hypoelliptic operator. For this class of operators, we shall
prove the following global $L^{p}$ estimates:

\begin{theorem}
\label{Thm main}For every $p\in\left(  1,\infty\right)  $ there exists a
constant $c>0,$ depending on $p,N,p_{0}$, the matrix $B,$ the number $\Lambda$
in (\ref{ellitpicity}) and the modulus $\omega$ in (\ref{w}) such that for
every $u\in C_{0}^{\infty}\left(  \mathbb{R}^{N}\right)  $ one has:
\begin{align}
\sum_{i,j=1}^{p_{0}}\left\Vert \partial_{x_{i}x_{j}}^{2}u\right\Vert
_{L^{p}\left(  \mathbb{R}^{N}\right)  }  &  \leq c\left\{  \left\Vert
\mathcal{A}u\right\Vert _{L^{p}\left(  \mathbb{R}^{N}\right)  }+\left\Vert
u\right\Vert _{L^{p}\left(  \mathbb{R}^{N}\right)  }\right\}
,\label{stima L_0 a}\\
\left\Vert \sum_{i,j=1}^{N}b_{ij}x_{i}\partial_{x_{j}}u\right\Vert
_{L^{p}\left(  \mathbb{R}^{N}\right)  }  &  \leq c\left\{  \left\Vert
\mathcal{A}u\right\Vert _{L^{p}\left(  \mathbb{R}^{N}\right)  }+\left\Vert
u\right\Vert _{L^{p}\left(  \mathbb{R}^{N}\right)  }\right\}  .
\label{stima L_0 b}%
\end{align}

\end{theorem}

In \cite{BCLP} we have proved this result in the case of constant coefficients
$a_{ij}$. Here we show that exploiting results and techniques contained in
\cite{BCLP}, together with a careful inspection of the quantitative dependence
of some bounds proved in \cite{LP} and \cite{DP}, we can get Theorem
\ref{Thm main}. The striking feature of our result is twofold. On the one
side, the merely uniform continuity of the coefficients $a_{ij}(x)$; on the
other side the lack of a Lie group structure making translation invariant the
frozen operator
\[
\mathcal{A}_{x_{0}}=\sum_{i,j=1}^{p_{0}}a_{ij}\left(  x_{0}\right)
\partial_{x_{i}x_{j}}^{2}+\sum_{i,j=1}^{N}b_{ij}x_{i}\partial_{x_{j}},\qquad
x_{0}\in\mathbb{R}^{N}.
\]
As in \cite{BCLP}, we overcome this last difficulty by considering the
operator $\mathcal{A}$ as the stationary counterpart of the corresponding
evolution operator $\mathcal{A}-\partial_{t}$ and looking for the estimates
\eqref{stima L_0 a} and \eqref{stima L_0 b} as a consequence of analogous
estimates for $\mathcal{A}-\partial_{t}$ on a suitable strip $S_{T}%
=\mathbb{R}^{N}\times\left[  -T,T\right]  $.

 There exists a quite extensive literature related to global
$L^p$ estimates for non-degenerate elliptic and parabolic equations
on the whole space with unbounded lower order coefficients and
variable coefficients $a_{ij}$. The considered $L^p$-spaces are
defined with respect to Lebesgue measure or with respect to an
invariant measure which has  also a probabilistic interpretation
(see, for instance, \cite{CV}, \cite{CF}, \cite{DG}, \cite{K},
\cite{FF}, \cite{MPSR}, \cite{PRS} and the references therein).

On the other hand,  to the best of our knowledge, only
the papers \cite{BB1}, \cite{BCM} and \cite{BZ} deal with $L^p$ estimates for classes of {\em degenerate}
   operators with both
unbounded first order coefficients and bounded variable coefficients
$a_{ij}$. However, we want to stress that the estimates there proved are only of {\em local} type.

   We also  mention that  global $L^p$ estimates like \eqref{stima L_0 a}
    are crucial
in establishing  weak uniqueness theorems for associated stochastic
differential equations, see \cite{P} and the references therein.
 Finally,  a priori estimates in non-isotropic
  H\"older spaces for
 operators  like \eqref{ou1}   with H\"older
continuous $a_{ij}$ were proved by A. Lunardi in \cite{Lu}.

\section{Notations and preliminary results}

\subsection*{The operator $L$}

Let us introduce the evolution operator
\begin{align}
Lu(z)  &  =\sum_{i,j=1}^{p_{0}}a_{ij}\left(  z\right)  \partial_{x_{i}x_{j}%
}^{2}u(z)+\sum_{i,j=1}^{N}b_{ij}x_{i}\partial_{x_{j}}u(z)-\partial
_{t}u(z)\nonumber\\
&  =\sum_{i,j=1}^{p_{0}}a_{ij}\left(  z\right)  \partial_{x_{i}x_{j}}%
^{2}u(z)+\langle x,B\nabla u(z)\rangle-\partial_{t}u(z) \label{L}%
\end{align}
with $z=(x,t)$ in $\mathbb{R}^{N+1}$, where now the coefficients $a_{ij}$
possibly depend also on $t$. When the $a_{ij}$'s are time independent, we get
$L=\mathcal{A}-\partial_{t}$. Let%
\[
A(z)=%
\begin{bmatrix}
A_{0}(z) & 0\\
0 & 0
\end{bmatrix}
\]
be an $N\times N$ matrix where $A_{0}(z)=\left(  a_{ij}\left(  z\right)
\right)  _{i,j=1}^{p_{0}}$ is a $p_{0}\times p_{0}$ ($p_{0}\leq N$) symmetric
and uniformly positive definite matrix for all $z$, satisfying
\begin{equation}
\frac{1}{\Lambda}\left\vert \xi\right\vert ^{2}\leq\sum_{i,j=1}^{p_{0}}%
a_{ij}\left(  z\right)  \xi_{i}\xi_{j}\leq\Lambda\left\vert \xi\right\vert
^{2} \label{ellipticity2}%
\end{equation}
for all $\xi\in\mathbb{R}^{p_{0}}$ and for some constant $\Lambda\geq1$.

Moreover, we assume that the functions $a_{ij}$ are uniformly continuous in
$\mathbb{R}^{N+1}$ with modulus of continuity
\begin{equation}
\omega\left(  r\right)  =\max_{i,j=1,...,p_{0}}\sup_{\substack{z_{1},z_{2}%
\in\mathbb{R}^{N+1}\\\left\vert z_{1}-z_{2}\right\vert \leq r}}\left\vert
a_{ij}\left(  z_{1}\right)  -a_{ij}\left(  z_{2}\right)  \right\vert .
\label{ww}%
\end{equation}

\subsection*{The operator $L_{z_{0}}$}

For a fixed $z_{0}\in\mathbb{R}^{N+1}$ we consider the operator $L_{z_{0}}$
that differs from $L$ only for the coefficients $a_{ij}$'s, that now are
constant coefficients:
\[
L_{z_{0}}u(z)=\sum_{i,j=1}^{p_{0}}a_{ij}\left(  z_{0}\right)  \partial
_{x_{i}x_{j}}^{2}u(z)+\left\langle x,B\nabla u(z)\right\rangle -\partial
_{t}u(z),
\]
where, as above, $z=\left(  x,t\right)  $.

This operator is hypoelliptic; actually it can be proved (see \cite{LP}) that
this fact is equivalent to the validity of the condition $C(z_{0};t)>0$ for
every $t>0$, where
\[
C\left(  z_{0};t\right)  =\int_{0}^{t}E\left(  s\right)  A\left(
z_{0}\right)  E^{T}\left(  s\right)  ds\text{, where }E\left(  s\right)
=\exp\left(  -sB^{T}\right)  .
\]
Moreover, it is proved in \cite{LP} that $L_{z_{0}}$ is left-invariant with
respect to the composition law%
\[
\left(  x,t\right)  \circ\left(  \xi,\tau\right)  =\left(  \xi+E\left(
\tau\right)  x,t+\tau\right)  .
\]
Note that%
\[
\left(  \xi,\tau\right)  ^{-1}=\left(  -E\left(  -\tau\right)  \xi
,-\tau\right)  .
\]
We explicitly note that such a composition law is independent of $z_{0}$,
since only the matrix $B$ is involved.

The operator $L_{z_{0}}$ has a fundamental solution $\Gamma(z_{0};\cdot
,\cdot)$,
\[
\Gamma\left(  z_{0};z,\zeta\right)  =\gamma\left(  z_{0};\zeta^{-1}\circ
z\right)  \text{ for }z,\zeta\in\mathbb{R}^{N+1},
\]
with%
\begin{equation}
\gamma\left(  z_{0};z\right)  =\left\{
\begin{tabular}
[c]{ll}%
$0$ & $\text{for }t\leq0$\\
$\frac{\left(  4\pi\right)  ^{-N/2}}{\sqrt{\det C\left(  z_{0};t\right)  }%
}\exp\left(  -\frac{1}{4}\left\langle C^{-1}\left(  z_{0};t\right)
x,x\right\rangle -t\text{Tr}B\right)  $ & $\text{for }t>0$%
\end{tabular}
\ \ \right. \nonumber
\end{equation}
where $z=\left(  x,t\right)  $.

\subsection*{The operator $K_{z_{0}}$}

By \textit{principal part }of $L_{z_{0}}$ we mean the operator%
\[
K_{z_{0}}=\sum_{i,j=1}^{p_{0}}a_{ij}\left(  z_{0}\right)  \partial_{x_{i}%
x_{j}}^{2}+\left\langle x,B_{0}\nabla\right\rangle -\partial_{t},
\]
where the matrix in the drift term is now $B_{0}$, obtained by annihilating
every $\ast$ block in (\ref{B}):%
\[
B_{0}=%
\begin{bmatrix}
0 & B_{1} & 0 & \ldots & 0\\
0 & 0 & B_{2} & \ldots & 0\\
\vdots & \vdots & \vdots & \ddots & \vdots\\
0 & 0 & 0 & \ldots & B_{r}\\
0 & 0 & 0 & \ldots & 0
\end{bmatrix}
.
\]

The fundamental solution of the principal part operator $K_{z_{0}}$ is
$\Gamma_{0}\left(  z_{0};z,\zeta\right)  =\gamma_{0}\left(  z_{0};\zeta
^{-1}\circ z\right)  $; namely, for $t>0$%
\[
\gamma_{0}\left(  z_{0};z\right)  =\frac{\left(  4\pi\right)  ^{-N/2}}%
{\sqrt{\det C_{0}\left(  z_{0};t\right)  }}\exp\left(  -\frac{1}%
{4}\left\langle C_{0}^{-1}\left(  z_{0};t\right)  x,x\right\rangle \right)
\]
with
\begin{equation}
C_{0}\left(  z_{0};t\right)  =\int_{0}^{t}E_{0}\left(  s\right)  A(z_{0}%
)E_{0}^{T}\left(  s\right)  ds\text{, where }E_{0}\left(  s\right)
=\exp\left(  -sB_{0}^{T}\right)  . \label{C_0}%
\end{equation}

\subsubsection*{Homogeneous dimension, norm and distance}

For every $\lambda>0,$ let us define the matrix of \textit{dilations on
}$\mathbb{R}^{N},$%
\[
D\left(  \lambda\right)  =\operatorname{diag}\left(  \lambda I_{p_{0}}%
,\lambda^{3}I_{p_{1}},...,\lambda^{2r+1}I_{p_{r}}\right)
\]
where $I_{p_{j}}$ denotes the $p_{j}\times p_{j}$ identity matrix, and the
matrix of \textit{dilations} on $\mathbb{R}^{N+1}$,
\[
\delta\left(  \lambda\right)  =\left(  D(\lambda),\lambda^{2}\right)
=\operatorname{diag}\left(  \lambda I_{p_{0}},\lambda^{3}I_{p_{1}}%
,...,\lambda^{2r+1}I_{p_{r}},\lambda^{2}\right)  .
\]
Note that
\[
\det\left(  D\left(  \lambda\right)  \right)  =\lambda^{Q},\qquad\det\left(
\delta\left(  \lambda\right)  \right)  =\lambda^{Q+2}%
\]
with $Q=p_{0}+3p_{1}+...+\left(  2r+1\right)  p_{r}$; \ $Q$ and $Q+2$ are
called the \textit{homogeneous dimension} of $\mathbb{R}^{N}$ and
$\mathbb{R}^{N+1}$, respectively. The operator $K_{z_{0}}$ is homogeneous of
degree two with respect to these dilations.

There is a natural \textit{homogeneous norm} in $\mathbb{R}^{N+1},$ induced by
these dilations:%
\[
\left\Vert \left(  x,t\right)  \right\Vert =\sum_{j=1}^{N}\left\vert
x_{j}\right\vert ^{1/q_{j}}+\left\vert t\right\vert ^{1/2},
\]
where $q_{j}$ are positive integers such that $D\left(  \lambda\right)
=\operatorname{diag}\left(  \lambda^{q_{1}},...,\lambda^{q_{N}}\right)  .$
Clearly, we have%
\[
\left\Vert \delta\left(  \lambda\right)  z\right\Vert =\lambda\left\Vert
z\right\Vert \text{ \ \ for every }\lambda>0,z\in\mathbb{R}^{N+1}.
\]
A key geometrical object is the \textit{local quasisymmetric quasidistance}%
\ $d$. Namely,
\[
d\left(  z,\zeta\right)  =\left\Vert \zeta^{-1}\circ z\right\Vert .
\]
Note that the homogeneous norm involved in the definition of $d$ is related to
the principal part operator $K_{z_{0}}$, while the group law $\circ$ is
related to the original operator $L_{z_{0}}$. Hence this function $d$ is not a
usual quasidistance on a homogeneous group. The function $d\left(
z,\zeta\right)  $ satisfies the quasisymmetric and quasitriangle inequalities
only for $d\left(  z,\zeta\right)  $ bounded (see Lemma 2.1 in \cite{DP});
this happens for instance on a fixed $d$-ball $B_{\rho}(z)$, where
\[
B_{\rho}(z)=\left\{  \zeta\in\mathbb{R}^{N+1}\,:\,d\left(  z,\zeta\right)
<\rho\right\}  .
\]

\section{Estimates on a strip for evolution operators}

Let $S_{T}$ be the strip $\mathbb{R}^{N}\times\lbrack-T,T]$. We use $c$ to
denote constants that may vary from line to line.

Our main result in this section is the following:

\begin{theorem}
\label{Thm main thm strip} Let $L$ be as in (\ref{L}), with the matrix $B$
satisfying (\ref{B}) and with uniformly continuous coefficients $a_{ij}$
satisfying (\ref{ellipticity2}).

For every $p\in\left(  1,\infty\right)  $ there exist constants $c,T>0$
depending on $p,N,p_{0}$, the matrix $B$, the number $\Lambda$ in
(\ref{ellipticity2}), $c$ also depending on the modulus of continuity $\omega$
in (\ref{ww}) such that%
\begin{equation}
\sum_{i,j=1}^{p_{0}}\left\Vert \partial_{x_{i}x_{j}}^{2}u\right\Vert
_{L^{p}\left(  S_{T}\right)  }\leq c\left\{  \left\Vert Lu\right\Vert
_{L^{p}\left(  S_{T}\right)  }+\left\Vert u\right\Vert _{L^{p}\left(
S_{T}\right)  }\right\}  \label{bound strip}%
\end{equation}
for every $u\in C_{0}^{\infty}\left(  S_{T}\right)  $.
\end{theorem}

>From Theorem \ref{Thm main thm strip} one obtains Theorem \ref{Thm main}
proceeding as follows.

\bigskip

\begin{proof}
[Proof of Theorem \ref{Thm main}]If $u:\mathbb{R}^{N}\rightarrow\mathbb{R}$ is
a $C_{0}^{\infty}$ function, we define
\[
U\left(  x,t\right)  =u\left(  x\right)  \psi\left(  t\right)  ,
\]
where
\[
\psi\in C_{0}^{\infty}\left(  \mathbb{R}\right)
\]
is a cutoff function with sprt$\,\psi\subset\left[  -T,T\right]  $,
\ $\int_{-T}^{T}\psi\left(  t\right)  dt>0$. Then (\ref{bound strip}) applied
to $U$ gives (\ref{stima L_0 a}) for $u$. Moreover, inequality
\eqref{stima L_0 b} immediately follows by difference.
\end{proof}

The crucial
step toward the proof of Theorem \ref{Thm main thm strip} is a
local estimate contained in the following:

\begin{proposition}
\label{Prop local bound} There exist constants $c,r_{0}$ such that for every
$z_{0}\in S_{T}$, $r\leq r_{0}$, $u\in C_{0}^{\infty}\left(  B_{r}\left(
z_{0}\right)  \right)  $, we have%
\begin{equation}
\sum_{i,j=1}^{p_{0}}\left\Vert \partial_{x_{i}x_{j}}^{2}u\right\Vert
_{L^{p}\left(  B_{r}\left(  z_{0}\right)  \right)  }\leq c\left\Vert
Lu\right\Vert _{L^{p}\left(  B_{r}\left(  z_{0}\right)  \right)  }\text{.}
\label{local bound}%
\end{equation}

\end{proposition}

\begin{proof}
Let $z_{0}\in S_{T}$ and $\rho_{0}\in(0,T]$ be fixed and choose a cutoff
function \ $\eta\in C_{0}^{\infty}\left(  \mathbb{R}^{N+1}\right)  $ such
that
\begin{align*}
\eta\left(  z\right)   &  =1\text{ for }\left\Vert z\right\Vert \leq\rho
_{0}/2;\\
\eta\left(  z\right)   &  =0\text{ for }\left\Vert z\right\Vert \geq\rho_{0}.
\end{align*}
Then, by \cite[Proposition 2.11]{DP} and (25) in \cite{BCLP}, we have, for
every $u\in C_{0}^{\infty}\left(  B_{r}\left(  z_{0}\right)  \right)  ,$
\begin{align}
\partial_{x_{i}x_{j}}^{2}u=  &  -PV\left(  L_{z_{0}}u\ast\left(  \eta
\partial_{x_{i}x_{j}}^{2}\gamma\left(  z_{0};\cdot\right)  \right)  \right)
-L_{z_{0}}u\ast\left(  \left(  1-\eta\right)  \partial_{x_{i}x_{j}}^{2}%
\gamma\left(  z_{0};\cdot\right)  \right) \nonumber\\
&  +c_{ij}\left(  z_{0}\right)  L_{z_{0}}u\nonumber\\
\equiv &  -PV\left(  L_{z_{0}}u\ast k_{0}\left(  z_{0};\cdot\right)  \right)
-L_{z_{0}}u\ast k_{\infty}\left(  z_{0};\cdot\right)  +c_{ij}\left(
z_{0}\right)  L_{z_{0}}u \label{repr formula}%
\end{align}
having set:%
\begin{equation}
\begin{aligned} k_{0}\left( z_{0};\cdot\right) & =\eta\partial_{x_{i}x_{j}}^{2}\gamma\left( z_{0};\cdot\right)\\ k_{\infty}\left( z_{0};\cdot\right) & =\left( 1-\eta\right) \partial_{x_{i}x_{j}}^{2}\gamma\left( z_{0};\cdot\right) \end{aligned} \label{k}%
\end{equation}
and
\[
c_{ij}(z_{0})=-\int_{\Vert\zeta\Vert=1}\partial_{x_{i}}\gamma_{0}(z_{0}%
;\zeta)\nu_{j}(\zeta)\,d\sigma(\zeta),
\]
where $\nu_{j}$ denotes the $j$-th component of the exterior normal $\nu$ to
the boundary of $\{\Vert\zeta\Vert<1\}$. In \eqref{repr formula} $\ast$
denotes the convolution with respect to the composition law $\circ$.

Writing%
\begin{align*}
L_{z_{0}}u\left(  z\right)   &  =\left(  L_{z_{0}}-L\right)  u\left(
z\right)  +Lu\left(  z\right) \\
&  =\sum_{i,j=1}^{p_{0}}\left(  a_{ij}\left(  z_{0}\right)  -a_{ij}\left(
z\right)  \right)  \partial_{x_{i}x_{j}}^{2}u\left(  z\right)  +Lu\left(
z\right)
\end{align*}
we get, by \eqref{repr formula},
\begin{align}
\partial_{x_{i}x_{j}}^{2}u=  &  -PV\left(  Lu\ast k_{0}\left(  z_{0}%
;\cdot\right)  \right)  -PV\left(  \sum_{h,k=1}^{p_{0}}\left(  a_{hk}\left(
z_{0}\right)  -a_{hk}\left(  \cdot\right)  \right)  \partial_{x_{h}x_{k}}%
^{2}u\ast k_{0}\left(  z_{0};\cdot\right)  \right) \nonumber\\
&  -Lu\ast k_{\infty}\left(  z_{0};\cdot\right)  -\sum_{h,k=1}^{p_{0}}\left(
a_{hk}\left(  z_{0}\right)  -a_{hk}\left(  \cdot\right)  \right)
\partial_{x_{h}x_{k}}^{2}u\ast k_{\infty}\left(  z_{0};\cdot\right)
\nonumber\\
&  +c_{ij}\left(  z_{0}\right)  Lu+c_{ij}\left(  z_{0}\right)  \sum
_{h,k=1}^{p_{0}}\left(  a_{hk}\left(  z_{0}\right)  -a_{hk}\left(
\cdot\right)  \right)  \partial_{x_{h}x_{k}}^{2}u\nonumber\\
=  &  I_{1}+I_{2}+J_{1}+J_{2}+A_{1}+A_{2}. \label{repr formula var}%
\end{align}

We now split the remaining part of the proof into several steps.

\medbreak
\textbf{Step 1.} \emph{$L^{p}$-estimate of $A_{1}$ and $A_{2}$.}

We obviously have
\[
\Vert A_{1}\Vert_{L^{p}(B_{r}(z_{0}))}\leq|c_{ij}(z_{0})|\Vert Lu\Vert
_{L^{p}(B_{r}(z_{0}))}.
\]
On the other hand, by Theorem \ref{DP-Prop2.7} and Remark \ref{remDP-Prop2.7}
in Appendix, there exists an absolute constant $c$ such that
\[
|c_{ij}(z_{0})|\leq\int_{\Vert\zeta\Vert=1}\left\vert \partial_{x_{i}}%
\gamma_{0}(z_{0};\zeta)\right\vert \,d\sigma(\zeta)\leq c\int_{\Vert\zeta
\Vert=1}\frac{1}{\Vert\zeta\Vert^{Q+1}}\,d\sigma(\zeta).
\]
Therefore
\begin{equation}
\Vert A_{1}\Vert_{L^{p}(B_{r}(z_{0}))}\leq c\Vert Lu\Vert_{L^{p}(B_{r}%
(z_{0}))}. \label{stima1}%
\end{equation}

Analogously, using the uniform continuity of the coefficients $a_{ij}$'s, we
get
\begin{equation}
\|A_{2}\|_{L^{p}(B_{r}(z_{0}))}\le c\,\omega(r)\sum_{h,k=1}^{p_{0}}
\|\partial^{2}_{x_{h}x_{k}}u \|_{L^{p}(B_{r}(z_{0}))}. \label{stima2}%
\end{equation}

\medbreak
\textbf{Step 2.} \emph{$L^{p}$-estimate of $J_{1}$ and $J_{2}$.}

Without loss of generality we can assume $B_{r}(z_{0})\subseteq S_{2T}$ for
every $z_{0}\in S_{T}$. Then
\[
\Vert J_{1}\Vert_{L^{p}(B_{r}(z_{0}))}\leq c\int_{S_{2T}}|k_{\infty}%
(z_{0};\zeta)|\,d\zeta\,\Vert Lu\Vert_{L^{p}(B_{r}(z_{0}))},
\]
where the presence of the constant $c$ depends on the fact that our group is
not unimodular. On the other hand, just proceeding as in \cite{BCLP}, pages
799-800, and using the estimates in Appendix (see Proposition \ref{propapp}) we get
\[
\int_{S_{2T}}|k_{\infty}(z_{0};\zeta)|\,d\zeta\leq c,
\]
where $c$ is independent of $z_{0}\in S_{T}$. Therefore
\begin{equation}
\Vert J_{1}\Vert_{L^{p}(B_{r}(z_{0}))}\leq c\Vert Lu\Vert_{L^{p}(B_{r}%
(z_{0}))}.\label{stima3}%
\end{equation}
Analogously, using the uniform continuity of the $a_{ij}$'s, we get
\begin{equation}
\Vert J_{2}\Vert_{L^{p}(B_{r}(z_{0}))}\leq c\omega(r)\sum_{h,k=1}^{p_{0}}%
\Vert\partial_{x_{h}x_{k}}^{2}u\Vert_{L^{p}(B_{r}(z_{0}))}.\label{stima4}%
\end{equation}

\medbreak
\textbf{Step 3.} \emph{$L^{p}$-estimate of $I_{1}$ and $I_{2}$.}

To estimate the $L^{p}$-norm of $I_{1}$ and $I_{2}$, we can use Theorem
\ref{th5.?}, getting:
\begin{align}
&  \Vert I_{1}\Vert_{L^{p}(B_{r}(z_{0}))}+\Vert I_{2}\Vert_{L^{p}(B_{r}%
(z_{0}))}\nonumber\\
&  \leq c\left\{  \Vert Lu\Vert_{L^{p}(B_{r}(z_{0}))}+\left\Vert \sum
_{h,k=1}^{p_{0}}\left[  a_{hk}\left(  z_{0}\right)  -a_{hk}\left(
\cdot\right)  \right]  \partial_{x_{h}x_{k}}^{2}u\right\Vert _{L^{p}\left(
B_{r}\left(  z_{0}\right)  \right)  }\right\}  \nonumber\\
&  \leq c\left\{  \left\Vert Lu\right\Vert _{L^{p}\left(  B_{r}\left(
z_{0}\right)  \right)  }+\omega\left(  r\right)  \sum_{h,k=1}^{p_{0}%
}\left\Vert \partial_{x_{h}x_{k}}^{2}u\right\Vert _{L^{p}\left(  B_{r}\left(
z_{0}\right)  \right)  }\right\}  \label{stima5}%
\end{align}
with $c$ independent of $r$ and $z_{0}$.

\medbreak
\textbf{Step 4.} \emph{Conclusion.}

By \eqref{repr formula var} and the estimates \eqref{stima1}-\eqref{stima5} in
the previous steps, we get
\[
\sum_{i,j=1}^{p_{0}}\left\Vert \partial_{x_{i}x_{j}}^{2}u\right\Vert
_{L^{p}\left(  B_{r}\left(  z_{0}\right)  \right)  }\leq c\left\{  \left\Vert
Lu\right\Vert _{L^{p}\left(  B_{r}\left(  z_{0}\right)  \right)  }%
+\omega\left(  r\right)  \sum_{h,k=1}^{p_{0}}\left\Vert \partial_{x_{h}x_{k}%
}^{2}u\right\Vert _{L^{p}\left(  B_{r}\left(  z_{0}\right)  \right)
}\right\}
\]
with $c$ independent of $r$ and $z_{0}$.

\noindent We now fix once and for all $r_{0}$ small enough so that
$c\omega\left(  r_{0}\right)  <1$, getting%
\[
\sum_{i,j=1}^{p_{0}}\left\Vert \partial_{x_{i}x_{j}}^{2}u\right\Vert
_{L^{p}\left(  B_{r}\left(  z_{0}\right)  \right)  }\leq c\left\Vert
Lu\right\Vert _{L^{p}\left(  B_{r}\left(  z_{0}\right)  \right)  }%
\]
for every $u\in C_{0}^{\infty}\left(  B_{r}(z_{0}) \right)  $ with $r\leq
r_{0}$, with $c,r_{0}$ independent of $u$ and $z_{0}\in S_{T}$.
\end{proof}

Next, we have to prove the following crucial
ingredient which has been used in
the previous proof:

\bigskip

\begin{theorem}
\label{th5.?}Let $k_{0}(z_{0};\cdot)$ be the singular kernel defined
in \eqref{k}. For every $p\in(1,\infty)$ there exists a positive
constant $c$, independent of $z_{0}$, such that
\begin{equation}
\Vert PV(f\ast k_{0}(z_0; \cdot)) \Vert_{L^{p}(B_{r}(z_{0}))}\leq
c\Vert f\Vert _{L^{p}(B_{r}(z_{0}))}\label{a}
\end{equation}
for every $f\in C_{0}^{\infty}(B_{r}(z_{0}))$, $z_{0}\in S_{T}$ and $r>0$ such
that $B_{r}(z_{0})\subseteq S_{2T}$.
\end{theorem}

\begin{proof}
This theorem is analogous to Theorem 22 in \cite{BCLP}, the novelty being the
uniformity of the bound with respect to the point $z_{0}$ in the kernel
$k_{0}(z_{0};\cdot)$. As in \cite{BCLP}, this theorem follows applying the
abstract result contained in \cite[Thm 3]{B}. Without recalling the general
setting of nondoubling spaces considered in \cite{B}, here we just list, for
convenience of the reader, the assumptions that need to be checked on our
kernel, in order to derive Theorem \ref{th5.?} from \cite[Thm 3]{B}. The
constant $c$ in (\ref{a}) will depend only on the constants involved in the
following bounds.

Let%
\[
k(z_{0};w^{-1}\circ z)=a\left(  z\right)  k_{0}(z_{0};w^{-1}\circ z)b\left(
w\right)
\]
where $a,b\in C_{0}^{\infty}\left(  \mathbb{R}^{N+1}\right)  $ with $\operatorname{sprt}a$,
$\operatorname{sprt}b\subset B_{r}\left(  z_{0}\right)$. Then the required properties are
the following:%
\begin{equation}
\left\vert k(z_{0};w^{-1}\circ z)\right\vert +\left\vert k(z_{0};z^{-1}\circ
w)\right\vert \leq\frac{c}{\Vert w^{-1}\circ z\Vert^{Q+2}}\label{standard 1}%
\end{equation}
for every $z_{0}\in S_{T},$ $z,w\in S_{2T}$ such that $\left\Vert w^{-1}\circ
z\right\Vert \leq1;$%
\begin{align}
& \left\vert k(z_{0};w^{-1}\circ z)-k(z_{0};w^{-1}\circ\overline
{z})\right\vert +\label{standard 2b}\\
& \left\vert k(z_{0};z^{-1}\circ w)-k(z_{0};\overline{z}^{-1}\circ
w)\right\vert \leq c\frac{\Vert z^{-1}\circ\overline{z}\Vert}{\Vert
w^{-1}\circ z\Vert^{Q+3}}\nonumber
\end{align}
for every $z_{0}\in S_{T},z,\overline{z},w\in S_{2T}$ such that $\Vert
z^{-1}\circ\overline{z}\Vert\leq M\Vert w^{-1}\circ z\Vert$ and $\left\Vert
w^{-1}\circ z\right\Vert \leq1;$%
\begin{equation}
\left\vert \int_{r_{1}\leq\Vert\zeta^{-1}\circ z\Vert\leq r_{2}}k(z_{0}%
;\zeta^{-1}\circ z)\,d\zeta\right\vert +\left\vert \int_{r_{1}\leq\Vert
\zeta^{-1}\circ z\Vert\leq r_{2}}k(z_{0};z^{-1}\circ\zeta)\,d\zeta\right\vert
\leq c\label{boundedness}%
\end{equation}
for every $r_{1},r_{2}$ with $0<r_{1}<r_{2}$ and for all $z\in S_{2T}$ and
$z_{0}\in S_{T};$%
\begin{equation}
h\left(  z_{0},\cdot\right)  -h^{\ast}\left(  z_{0},\cdot\right)  \in
C^{\gamma}\left(  B_{r}(z_{0})\right)  \label{holder}%
\end{equation}
for some positive $\gamma$, where%
\begin{align}
h\left(  z_{0},z\right)    & =\lim_{r\rightarrow0}\int_{r\leq\Vert\zeta
^{-1}\circ z\Vert}k(z_{0};\zeta^{-1}\circ z)\,d\zeta;\label{T1}\\
h^{\ast}\left(  z_{0},z\right)    & =\lim_{r\rightarrow0}\int_{r\leq\Vert
\zeta^{-1}\circ z\Vert}k(z_{0};z^{-1}\circ\zeta)\,d\zeta.\label{T*1}%
\end{align}

Now: estimates (\ref{standard 1}) and (\ref{standard 2b}) follow from Theorem
\ref{Thm main appendix} and Remark \ref{remDP-Prop2.7} contained in the
Appendix.

Let us prove (\ref{boundedness}). Actually, we will bound the first integral,
the bound on the second being analogous. Moreover, we actually prove the
following%
\begin{equation}
\left\vert \int_{r_{1}\leq\Vert\zeta^{-1}\circ z\Vert\leq r_{2}}k_{0}%
(z_{0};\zeta^{-1}\circ z)\,d\zeta\right\vert \leq c,\label{boundedness 2}%
\end{equation}
which implies the analogous bound on $k$ by the same argument contained in
\cite[Prop. 18]{BCLP}. To show (\ref{boundedness 2}), we proceed as in
\cite{BCLP}, page 803. Without loss of generality we assume $r_{2}\leq\rho
_{0}$, where $\rho_{0}$ is the positive constant introduced at the beginning
of the proof of Proposition \ref{Prop local bound}; in fact, $k_{0}%
(z_{0};w)=0$ for $\Vert w\Vert>\rho_{0}$.

We have:
\[
\int_{r_{1}\leq\Vert\zeta^{-1}\circ z\Vert\leq r_{2}}k_{0}(z_{0};\zeta
^{-1}\circ z)\,d\zeta=A(z_{0};r_{1},r_{2})+B(z_{0};r_{1},r_{2}),
\]
where
\[
A(z_{0};r_{1},r_{2})=\int_{r_{1}\leq\Vert w\Vert\leq r_{2}}\eta(w)\partial
_{x_{i}x_{j}}^{2}\gamma(z_{0};w)\,dw
\]
and
\[
B(z_{0};r_{1},r_{2})=\int_{r_{1}\leq\Vert w\Vert\leq r_{2}}\eta(w)\partial
_{x_{i}x_{j}}^{2}\gamma(z_{0};w)\left(  e^{\tau Tr(B)}-1\right)  \,dw,\qquad
w=(\xi,\tau).
\]
Then, by \eqref{DP2.37}
\[
B(z_{0};r_{1},r_{2})\leq c\,\int_{r_{1}\leq\Vert w\Vert\leq r_{2}}\frac
{1}{\Vert w\Vert^{Q}}\,dw\leq c\int_{\Vert w\Vert\leq\rho_{0}}\frac{1}{\Vert
w\Vert^{Q}}\,dw
\]
with $c$ independent of $z_{0}\in S_{T}$. Moreover, if $r_{2}\leq\frac
{\rho_{0}}{2}$, then integrating by parts
\begin{align*}
A(z_{0};r_{1},r_{2}) &  =\int_{\Vert w\Vert=r_{2}}\partial_{x_{i}}\gamma
(z_{0};w)\nu_{j}\,d\sigma(w)-\int_{\Vert w\Vert=r_{1}}\partial_{x_{i}}%
\gamma(z_{0};w)\nu_{j}\,d\sigma(w)\\
&  =:I(z_{0};r_{2})-I(z_{0};r_{1}).
\end{align*}
Now we estimate $I(z_{0};\rho)$ by proceeding as in \cite{DP}, page 1280. We
have
\begin{align*}
I(z_{0};\rho)= &  \int_{\Vert\zeta\Vert=1}\partial_{x_{i}}\gamma_{\rho}%
(z_{0};\zeta)\nu_{j}\,d\sigma(\zeta)\\
= &  \int_{\Vert\zeta\Vert=1}\left(  \partial_{x_{i}}\gamma_{\rho}(z_{0}%
;\zeta)-\partial_{x_{i}}\gamma_{0}(z_{0};\zeta)\right)  \nu_{j}\,d\sigma
(\zeta)\\
&  +\int_{\Vert\zeta\Vert=1}\partial_{x_{i}}\gamma_{0}(z_{0};\zeta)\nu
_{j}\,d\sigma(\zeta)
\end{align*}
where $\gamma_{\rho}(z_{0};\cdot)$ is defined as in \cite{DP}, (2.24).

The last integrand can be estimated by a constant independent of $z_{0}\in
S_{T}$, thanks to \eqref{DP2.36} and Remark \ref{remDP-Prop2.7}. On the other hand, from (2.45) in \cite{DP}
we get, for a suitable $c$ independent of $z_{0}\in S_{T}$:
\[
\left\vert \int_{\Vert\zeta\Vert=1}\left(  \partial_{x_{i}}\gamma_{\rho}%
(z_{0};\zeta)-\partial_{x_{i}}\gamma_{0}(z_{0};\zeta)\right)  \nu_{j}%
\,d\sigma(\zeta)\right\vert \leq c\,\rho\int_{\Vert\zeta\Vert=1}\frac{1}%
{\sqrt{\tau}}\gamma(\zeta)\,d\sigma(\zeta),
\]
$\zeta=(x,\tau)$, where $\gamma$ is the fundamental solution with pole at the
origin of
\[
\mu\sum_{i=1}^{p_{0}}\partial_{x_{i}}^{2}+\langle x,B_{0}\nabla\rangle
-\partial_{t}%
\]
for a suitable $\mu>0$ independent of $z_{0}\in S_{T}$. Note that the last
integral is an absolute constant.

Suppose now $\frac{\rho_{0}}{2}\leq r_{2}\leq\rho_{0}$. Then we can write
\[
A\left(  z_{0};r_{1},r_{2}\right)  \leq\left\vert \int_{r_{1}<\left\Vert
w\right\Vert <\rho_{0}/2}k_{0}\left(  z_{0};w\right)  dw\right\vert
+\left\vert \int_{\rho_{0}/2<\left\Vert w\right\Vert <r_{2}}k_{0}\left(
z_{0};w\right)  dw\right\vert .
\]
The first term can be bounded as above, while the second one is bounded by%
\[
\int_{\rho_{0}/2\leq\left\Vert w\right\Vert \leq\rho_{0}}c\left\Vert
w\right\Vert ^{-\left(  2+Q\right)  }\,dw
\]
with $c$ independent of $z_{0}$, see \eqref{DP2.37}. This completes the proof
of (\ref{boundedness 2}).

Finally, let us prove the H\"{o}lder continuity of the function
\[
h\left(  z_{0},\cdot\right)  -h^{\ast}\left(  z_{0},\cdot\right)
\]
defined in (\ref{T1})-(\ref{T*1}).\footnote{We take this opportunity to notice that
in \cite{BCLP} this check has not been explicitly done.}
\begin{align*}
h\left(  z_{0},z\right)    & =\lim_{r\rightarrow0}a\left(  z\right)
\int_{r\leq\Vert\zeta^{-1}\circ z\Vert}k_{0}(z_{0};\zeta^{-1}\circ z)b\left(
\zeta\right)  \,d\zeta=\\
& =\lim_{r\rightarrow0}a\left(  z\right)  \int_{r\leq\Vert
w\Vert}k_{0}
(z_{0};w)b\left(  z\circ w^{-1}\right)  \,e^{\tau Tr(B)}dw\\
& =a\left(  z\right)
 \int_{\Vert w\Vert \le \rho_0}
k_{0}(z_{0};w)\left[ b\left(  z\circ
w^{-1}\right)  \,-b\left(  z\right)  \right]  e^{\tau Tr(B)}dw+\\
& +a\left(  z\right)  b\left(  z\right)
\int_{\Vert w\Vert \le \rho_0}
 k_{0}(z_{0};w)\left[  e^{\tau
Tr(B)}-1\right]  dw\\
& +\lim_{r\rightarrow0}a\left(  z\right)  b\left(  z\right)  \int_{r\leq\Vert
w\Vert}k_{0}(z_{0};w)dw\\
& =h_{1}\left(  z_{0},z\right)  +h_{2}\left(  z_{0},z\right)  +h_{3}\left(
z_{0},z\right)  .
\end{align*}
Now:%
\[
h_{3}\left(  z_{0},z\right)  =a\left(  z\right)  b\left(  z\right)  c\left(
z_{0}\right)
\]
with $a\left(  \cdot\right)$, $b\left(  \cdot\right)  $ smooth and $c\left(
z_{0}\right)  $ uniformly bounded in $z_{0}$ by the previous bound
(\ref{boundedness 2}). Also,%
\[
h_{2}\left(  z_{0},z\right)  =a\left(  z\right)  b\left(  z\right)
c_{1}\left(  z_{0}\right)
\]
with $c\left(  z_{0}\right)  $ uniformly bounded in $z_{0}$ by the same
argument used above to bound $B(z_{0};r_{1},r_{2})$. Let us come to
$h_{1}\left(  z_{0},z\right)  $. If $Z$ is any right-invariant differential
operator, then%
\begin{align*}
Zh_{1}\left(  z_{0},z\right)    & =Za\left(  z\right)
\int_{\Vert w\Vert \le \rho_0}
 k_{0}%
(z_{0};w)
\left[  b\left(  z\circ w^{-1}\right)  \,-b\left(  z\right)  \right]
e^{\tau Tr(B)}dw\\
& +a\left(  z\right)
\int_{\Vert w\Vert \le \rho_0}
 k_{0}(z_{0};w)\left[  Zb\left(  z\circ
w^{-1}\right)  \,-Zb\left(  z\right)  \right]  e^{\tau Tr(B)}dw,
\end{align*}
hence%
\[
\left\vert Zh_{1}\left(  z_{0},z\right) \right\vert \leq c
 \int_{\Vert w\Vert \le \rho_0}
 \left\vert k_{0}(z_{0};w)\right\vert \left\vert w\right\vert dw\leq
c.
\]
Since this procedure can be iterated, we get an upper bound on any derivative
of the kind $\left\vert Z_{1}Z_{2}...Z_{k}h_{1}\left(  z_{0},z\right)
\right\vert $, hence (since the commutators of suitable right invariant vector
fields span $\mathbb{R}^{N+1}$) also on $\left\vert \nabla h_{1}\left(
z_{0},z\right)  \right\vert $. Therefore the function $h_{1}\left(
z_{0},\cdot\right)  $ is Lipschitz continuous, uniformly with respect to
$z_{0}$. The function $h^{\ast}\left(  z_{0},\cdot\right)  $ can be handled
similarly. This completes the proof of the conditions which are sufficient to
apply \cite[Thm 3]{B} and deduce (\ref{a}), with a constant $c$ independent of
$z_{0}$.
\end{proof}

In order to deduce Theorem \ref{Thm main thm strip} from Proposition
\ref{Prop local bound}, we now need to recall a covering lemma, see Lemma 21
in \cite{BCLP} (note that this result is not standard since our space is not
globally doubling):

\begin{lemma}
\label{Lemma Covering}For every $r_{0}>0$ and $K>1$ there exist $r\in\left(
0,r_{0}\right)  $, a positive integer $M$ and a sequence of points $\left\{
z_{i}\right\}  _{i=1}^{\infty}\subset S_{T}$ such that:%
\begin{align}
S_{T}  &  \subset\bigcup\limits_{i=1}^{\infty}B_{r}\left(  z_{i}\right)
;\label{cov 1}\\
\sum_{i=1}^{\infty}\chi_{B_{Kr}\left(  z_{i}\right)  }\left(  z\right)   &
\leq M\text{ \ }\forall z\in S_{T}. \label{cov 2}%
\end{align}

\end{lemma}

\begin{proof}
[Proof of Theorem \ref{Thm main thm strip}]Let us apply the previous lemma
with $r_{0}$ as in Proposition \ref{Prop local bound}; for a fixed
$r\in(0,r_{0})$, with $r/2$ satisfying (\ref{cov 1}), (\ref{cov 2}). Pick
$A\in C_{0}^{\infty}\left(  B_{r}\left(  0\right)  \right)  ,$ $A=1$ in
$B_{r/2}\left(  0\right)  ,$ $0\leq A\leq1$ and let $a_{k}\left(  z\right)
=A\left(  z_{k}^{-1}\circ z\right)  $.

Let now $u\in C_{0}^{\infty}\left(  S_{T}\right)  $. By (\ref{cov 1}) we can
write%
\begin{align}
\left\Vert \partial_{x_{i}x_{j}}^{2}u\right\Vert _{L^{p}\left(  S_{T}\right)
}^{p} &  \leq\sum_{k=1}^{\infty}\left\Vert \partial_{x_{i}x_{j}}%
^{2}u\right\Vert _{L^{p}\left(  B_{r/2}\left(  z_{k}\right)  \right)  }%
^{p}=\sum_{k=1}^{\infty}\left\Vert \partial_{x_{i}x_{j}}^{2}\left(
a_{k}u\right)  \right\Vert _{L^{p}\left(  B_{r/2}\left(  z_{k}\right)
\right)  }^{p}\nonumber\\
&  \leq\sum_{k=1}^{\infty}\left\Vert \partial_{x_{i}x_{j}}^{2}\left(
a_{k}u\right)  \right\Vert _{L^{p}\left(  B_{r}\left(  z_{k}\right)  \right)
}^{p}.\label{main 1}%
\end{align}
On the other hand, by (\ref{local bound}) we have%
\begin{align}
&  \left\Vert \partial_{x_{i}x_{j}}^{2}\left(  a_{k}u\right)  \right\Vert
_{L^{p}\left(  B_{r}\left(  z_{k}\right)  \right)  }\leq c\left\Vert L\left(
a_{k}u\right)  \right\Vert _{L^{p}\left(  B_{r}\left(  z_{k}\right)  \right)
}\nonumber\\
&  \leq c\left\{  \left\Vert a_{k}Lu\right\Vert _{L^{p}\left(  B_{r}\left(
z_{k}\right)  \right)  }+2\sum_{l,m=1}^{p_{0}}\left\Vert \partial_{x_{l}}%
a_{k}\partial_{x_{m}}u\right\Vert _{L^{p}\left(  B_{r}\left(  z_{k}\right)
\right)  }+\left\Vert uLa_{k}\right\Vert _{L^{p}\left(  B_{r}\left(
z_{k}\right)  \right)  }\right\}  .\label{main 2}%
\end{align}
By recalling that the operators $\partial_{x_{l}}$, $l=1,...,p_{0}$, and
$Y_{0}:=\sum_{i,j=1}^{N}b_{ij}x_{i}\partial_{x_{j}}$ are left invariant with
respect to the group law $\circ$, we have
\begin{align*}
\sup_{z\in B_{r}\left(  z_{k}\right)  }\left\vert \partial_{x_{l}}a_{k}\left(
z\right)  \right\vert  &  =\sup_{z\in B_{r}\left(  0\right)  }\left\vert
\partial_{x_{l}}A\left(  z\right)  \right\vert \leq c,\qquad l=1,...,p_{0},\\
\sup_{z\in B_{r}\left(  z_{k}\right)  }\left\vert Y_{0}a_{k}\left(  z\right)
\right\vert  &  =\sup_{z\in B_{r}\left(  0\right)  }\left\vert Y_{0}%
A(z)\right\vert \leq c
\end{align*}
and
\[
\sup_{z\in B_{r}\left(  z_{k}\right)  }\left\vert \partial_{x_{i}x_{j}}%
^{2}a_{k}\left(  z\right)  \right\vert =\sup_{z\in B_{r}\left(  0\right)
}\left\vert \partial_{x_{i}x_{j}}^{2}A\left(  z\right)  \right\vert \leq
c,\qquad i,j=1,2,...,p_{0}.
\]
As a consequence
\[
\sup_{z\in B_{r}\left(  z_{k}\right)  }\left\vert La_{k}\left(  z\right)
\right\vert \leq c
\]
with $c$ independent of $k$. Hence (\ref{main 2}) gives%
\[
\left\Vert \partial_{x_{i}x_{j}}^{2}\left(  a_{k}u\right)  \right\Vert
_{L^{p}\left(  B_{r}\left(  z_{k}\right)  \right)  }\leq c\left\{  \left\Vert
Lu\right\Vert _{L^{p}\left(  B_{r}\left(  z_{k}\right)  \right)  }%
+2\sum_{l,m=1}^{p_{0}}\left\Vert \partial_{x_{m}}u\right\Vert _{L^{p}\left(
B_{r}\left(  z_{k}\right)  \right)  }+\left\Vert u\right\Vert _{L^{p}\left(
B_{r}\left(  z_{k}\right)  \right)  }\right\}  ,
\]
$c$ independent of $k$. Inserting the last inequality in (\ref{main 1}) and
recalling (\ref{cov 2}) we get%
\begin{align*}
\left\Vert \partial_{x_{i}x_{j}}^{2}u\right\Vert _{L^{p}\left(  S_{T}\right)
}^{p} &  \leq c\sum_{k=1}^{\infty}\left\{  \left\Vert Lu\right\Vert
_{L^{p}\left(  B_{r}\left(  z_{k}\right)  \right)  }^{p}+\sum_{m=1}^{p_{0}%
}\left\Vert \partial_{x_{m}}u\right\Vert _{L^{p}\left(  B_{r}\left(
z_{k}\right)  \right)  }^{p}+\left\Vert u\right\Vert _{L^{p}\left(
B_{r}\left(  z_{k}\right)  \right)  }^{p}\right\}  \\
&  \leq cM\left\{  \left\Vert Lu\right\Vert _{L^{p}\left(  S_{T}\right)  }%
^{p}+\sum_{m=1}^{p_{0}}\left\Vert \partial_{x_{m}}u\right\Vert _{L^{p}\left(
S_{T}\right)  }^{p}+\left\Vert u\right\Vert _{L^{p}\left(  S_{T}\right)  }%
^{p}\right\}  .
\end{align*}
This also gives%
\[
\sum_{i,j=1}^{p_{0}}\left\Vert \partial_{x_{i}x_{j}}^{2}u\right\Vert
_{L^{p}\left(  S_{T}\right)  }\leq cM\left\{  \left\Vert Lu\right\Vert
_{L^{p}\left(  S_{T}\right)  }+\sum_{m=1}^{p_{0}}\left\Vert \partial_{x_{m}%
}u\right\Vert _{L^{p}\left(  S_{T}\right)  }+\left\Vert u\right\Vert
_{L^{p}\left(  S_{T}\right)  }\right\}
\]
which, by the classical interpolation inequality
\[
\left\Vert \partial_{x_{m}}u\right\Vert _{L^{p}\left(  S_{T}\right)  }%
\leq\varepsilon\left\Vert \partial_{x_{m}x_{m}}^{2}u\right\Vert _{L^{p}\left(
S_{T}\right)  }+\frac{c}{\varepsilon}\left\Vert u\right\Vert _{L^{p}\left(
S_{T}\right)  },
\]
yields (\ref{bound strip}). So we are done.
\end{proof}

\section{Appendix: uniform bounds on the fundamental solution of $L_{z_{0}}$}

The aim of this section is to prove the following result, which has been
exploited in the proof of Proposition \ref{Prop local bound} and Theorem
\ref{th5.?}:

\begin{theorem}
\label{Thm main appendix}There exists a positive constant $c$ independent of
$z_{0}\in S_{T}$ such that
\begin{align}
\left\vert \gamma(z_{0};\zeta)\right\vert  &  \leq\frac{c}{\Vert\zeta\Vert
^{Q}},\label{DP2.35}\\
\left\vert \partial_{x_{j}}\gamma(z_{0};\zeta)\right\vert  &  \leq\frac
{c}{\Vert\zeta\Vert^{Q+1}}\qquad j=1,...,p_{0},\label{DP2.36}\\
\left\vert \partial_{x_{i}x_{j}}^{2}\gamma(z_{0};\zeta)\right\vert  &
\leq\frac{c}{\Vert\zeta\Vert^{Q+2}}\qquad i,j=1,...,p_{0},\label{DP2.37}%
\end{align}
for every $\zeta\in S_{2T}$.

\noindent Morever, if $H\subset\mathbb{R}^{N}$ is a compact set there exist
constants $c^{\prime}$ and $M$, depending on $H$ but not on $z_{0}$, such
that
\begin{align}
\left\vert \gamma(z_{0};w^{-1}\circ z)-\gamma(z_{0};w^{-1}\circ\bar
{z})\right\vert  &  \leq c^{\prime}\frac{\Vert z^{-1}\circ\bar{z}\Vert}{\Vert
w^{-1}\circ z\Vert^{Q+1}},\nonumber\\
\left\vert \partial_{x_{j}}\gamma(z_{0};w^{-1}\circ z)-\partial_{x_{j}}%
\gamma(z_{0};w^{-1}\circ\bar{z})\right\vert  &  \leq c^{\prime}\frac{\Vert
z^{-1}\circ\bar{z}\Vert}{\Vert w^{-1}\circ z\Vert^{Q+2}}\qquad j=1,...,p_{0}%
,\nonumber\\
\left\vert \partial_{x_{i}x_{j}}^{2}\gamma(z_{0};w^{-1}\circ z)-\partial
_{x_{i}x_{j}}^{2}\gamma(z_{0};w^{-1}\circ\bar{z})\right\vert  &  \leq
c^{\prime}\frac{\Vert z^{-1}\circ\bar{z}\Vert}{\Vert w^{-1}\circ z\Vert^{Q+3}%
}\qquad i,j=1,...,p_{0},\label{standard 2}%
\end{align}
for every $z$,$\bar{z}$,$w\in S_{2T}$ such that $\Vert z^{-1}\circ\bar{z}%
\Vert\leq M\Vert w^{-1}\circ z\Vert$ and $w^{-1}\circ z\in H\times
\lbrack-2T,2T]$. \label{DP-Prop2.7}

\noindent
The previous estimates
 still hold replacing $\gamma(z_{0};z)$
 with $\gamma(z_{0};z^{-1})$.
\end{theorem}

\begin{remark}
\label{remDP-Prop2.7} The estimates of Theorem \ref{DP-Prop2.7} obviously hold
if we replace $\gamma(z_0;\cdot)$ with $\gamma_0(z_0;\cdot)$.
 In this case
we can exploit the homogeneity of $\gamma_{0}$ to obtain
\eqref{DP2.35}--\eqref{DP2.37} for every $\zeta$ in the strip $\mathbb{R}%
^{N}\times\lbrack-1,1]$.
\end{remark}

The above theorem will follow by a careful inspection of several arguments
contained in \cite{DP} and \cite{LP}. We first need to establish several lemmas.

\bigskip

In the following, $\mathcal{I}$ denotes the $N\times N$ matrix
\[
\mathcal{I}:=%
\begin{bmatrix}
I_{p_{0}} & 0\\
0 & 0
\end{bmatrix}
\]
where $I_{p_{0}}$ is the $p_{0}\times p_{0}$ identity matrix. Moreover, for
every $t>0$, $\widetilde{C}(t)$ is the $N\times N$ matrix defined as follows
\begin{equation}
\widetilde{C}(t)=\int_{0}^{t}E_{0}(s)\mathcal{I}E_{0}^{T}(s)\,ds\label{Ctilde}%
\end{equation}
with $E_{0}(s)$ as in (\ref{C_0}). Notice that $\widetilde{C}(t)>0$ for every
$t>0$, or, equivalently, that the operator
\[
\sum_{i=1}^{p_{0}}\partial_{x_{i}x_{i}}^{2}u(z)+\langle x,B_{0}\nabla
u(z)\rangle-\partial_{t}u(z)
\]
is hypoelliptic (see \cite{LP}).

\medbreak
The following preliminary lemma holds.

\begin{lemma}
\label{lemma1} The inequalities below hold true for all $z_{0}\in
\mathbb{R}^{N+1}$:
\begin{equation}
\frac{1}{\Lambda}\langle\widetilde{C}(1)y,y\rangle\leq\langle C_{0}%
(z_{0};1)y,y\rangle\leq\Lambda\langle\widetilde{C}(1)\,y,y\rangle\qquad\forall
y\in\mathbb{R}^{N} \label{stimaC_0}%
\end{equation}
and
\begin{equation}
\frac{1}{\Lambda^{N}}\operatorname{det}\widetilde{C}(1)\leq\operatorname{det}%
C_{0}(z_{0};1)\leq\Lambda^{N}\operatorname{det}\widetilde{C}(1).
\label{stimadetC_0}%
\end{equation}

\end{lemma}

\begin{proof}
We have that
\[
\frac{1}{\Lambda}\mathcal{I}\le A(z_{0})\le\Lambda\mathcal{I}\qquad\text{for
all $z_{0}\in\mathbb{R}^{N+1}$}.
\]
Thus, \eqref{stimaC_0} holds. \ Inequalities \eqref{stimadetC_0} are an easy
consequence of \eqref{stimaC_0}.
\end{proof}

\begin{lemma}
\label{lemma2}There exist $M\geq1,T>0$ such that for every $x\in\mathbb{R}%
^{N}$, $z_{0}\in\mathbb{R}^{N+1}$, $t\in\left[  0,T\right]  ,$%
\begin{equation}
\frac{1}{M}\langle\widetilde{C}(t)x,x\rangle\leq\langle C(z_{0};t)x,x\rangle
\leq M\langle\widetilde{C}(t)x,x\rangle\label{3.9LPnew}%
\end{equation}
and
\begin{equation}
\frac{1}{M}\operatorname{det}\widetilde{C}(t)\leq\operatorname{det}%
C(z_{0};t)\leq M\operatorname{det}\widetilde{C}(t). \label{stimadetC(t)}%
\end{equation}

\end{lemma}

\begin{proof}
It is a known fact (see \cite[Proposition 2.3]{LP}) that
\[
\begin{aligned}
C_{0}\left(z_0;  t\right)
&
=D({\sqrt{t}})  C_{0}\left(z_0;
1\right)  D({\sqrt{t}})
\\ \tilde{C}\left(t\right)
&
=D({\sqrt{t}})  \tilde{C}\left(1\right)  D( {\sqrt{t}}), \text{ \ \ \ }\forall \,t>0.
\end{aligned}
\]
Then \eqref{stimaC_0} implies
\begin{equation}
\frac{1}{\Lambda}\langle\widetilde{C}(t)x,x\rangle\leq\langle C_{0}%
(z_{0};t)x,x\rangle\leq\Lambda\langle\widetilde{C}(t)x,x\rangle
.\label{stimaC_0bis}%
\end{equation}
Therefore, to prove \eqref{3.9LPnew} it is enough to look for positive $c_{1}%
$, $c_{2}$ such that
\begin{equation}
\begin{aligned}c_1\left\langle C_0\left(z_0; t\right) x,x\right\rangle (1+O(t))&\le \left\langle C\left(z_0; t\right) x,x\right\rangle \\ &\le c_2\left\langle C_0\left(z_0; t\right) x,x\right\rangle (1+O(t))\qquad \quad\text{as $t\to 0$}\end{aligned}\label{3.9LP}%
\end{equation}
with $O(t)$ uniform w.r.t. $z_{0}$.

This follows using the arguments in \cite[p. 46]{LP}. Indeed, set $x=D\left(
\frac{1}{\sqrt{t}}\right)  y$ we get
\[
\begin{aligned}
\frac{\langle C\left(z_0;  t\right)  x,x\rangle}{\langle C_0\left(z_0;  t\right)  x,x\rangle}&=1+
\frac{\langle (C(z_0;t)-C_0(z_0;t))x,x\rangle}{\langle C_0(z_0;t)x,x\rangle}\\
&=1+\frac{\langle D\left(\frac{1}{\sqrt{t}}\right)(C(z_0;t)-C_0(z_0;t))D\left(\frac{1}{\sqrt{t}}\right) y,y\rangle}{\langle C_0(z_0;1)y,y\rangle}.
\end{aligned}
\]
Now, by the proof of Lemma 3.2 in \cite{LP} and a careful check of the block
decomposition of the matrices $C(z_{0};t)$ and $C_{0}(z_{0};t)$, see Lemma 3.1
in \cite{LP}, we get
\begin{equation}
\left\Vert D\left(  \frac{1}{\sqrt{t}}\right)  (C(z_{0};t)-C_{0}%
(z_{0};t))D\left(  \frac{1}{\sqrt{t}}\right)  \right\Vert \leq ct\quad\text{as
$t\rightarrow0^{+}$,}\label{M(t)}%
\end{equation}
uniformly w.r.t. $z_{0}$. Thus, we get \eqref{3.9LP}.

Let us now prove \eqref{stimadetC(t)}. By \eqref{stimaC_0bis}, we get
\[
\frac{1}{\Lambda^{N}}\operatorname{det}\widetilde{C}(t)\leq\operatorname{det}%
C_{0}(z_{0};t)\leq\Lambda^{N}\operatorname{det}\widetilde{C}(t).
\]
Moreover, by \eqref{3.9LP} there exist positive constants $c_{3}$, $c_{4}$
such that%
\[
c_{3}(1+O(t))\operatorname{det}C_{0}(z_{0};t)\leq\operatorname{det}%
C(z_{0};t)\leq c_{4}(1+O(t))\operatorname{det}C_{0}(z_{0};t)
\]
as $t$ goes to $0^{+}$, uniformly w.r.t. $z_{0}\in\mathbb{R}^{N+1}$. Thus,
\eqref{stimadetC(t)} follows.
\end{proof}

\medbreak
Now, we turn to prove estimates for $C^{-1}(z_{0};\cdot)$.

\begin{lemma}
The following inequalities hold:

\begin{itemize}
\item[(1)] there exist $M\geq1,T>0$ such that for every $x\in\mathbb{R}^{N}$,
$z_{0}\in\mathbb{R}^{N+1}$, $t\in\left[  0,T\right]  ,$%
\begin{equation}
\frac{1}{M}\langle C_{0}^{-1}(z_{0};t)x,x\rangle\leq\langle C^{-1}%
(z_{0};t)x,x\rangle\leq M\langle C_{0}^{-1}(z_{0};t)x,x\rangle
\label{3.10LPnew}%
\end{equation}

\item[(2)] let $\lambda_{\widetilde{{C}}}$ and $\Lambda_{\widetilde{{C}}}$ be
the smallest and the largest eigenvalue of the symmetric positive definite
matrix $\widetilde{{C}}(1)$, respectively. Then
\begin{equation}
\frac{1}{\Lambda\Lambda_{\widetilde{{C}}}}\left\vert D\left(  \frac{1}%
{\sqrt{t}}\right)  x\right\vert ^{2}\leq\langle C_{0}^{-1}(z_{0}%
;t)x,x\rangle\leq\frac{\Lambda}{\lambda_{\widetilde{{C}}}}\left\vert D\left(
\frac{1}{\sqrt{t}}\right)  x\right\vert ^{2}, \label{3.10BCLP2}%
\end{equation}
for all $x\in\mathbb{R}^{N}$ and for all $z_{0}\in\mathbb{R}^{N+1}$.
\end{itemize}

\label{lemma3}
\end{lemma}

\begin{proof}
The proof of \eqref{3.10LPnew} follows the lines of the proof of (3.10) in
\cite{LP}, using \eqref{M(t)} in place of (3.8) in \cite{LP}.

As far as \eqref{3.10BCLP2} it is concerned, we begin noticing that, see
\cite[p. 42]{LP},
\[
C_{0}^{-1}\left(  z_{0};t\right)  =D\left(  \frac{1}{\sqrt{t}}\right)
C_{0}^{-1}\left(  z_{0};1\right)  D\left(  \frac{1}{\sqrt{t}}\right)  ,\text{
\ \ \ }\forall\,t>0.
\]
Thus we have
\[
\begin{aligned}\langle C_0^{-1}(z_0;t)x,x\rangle&
=\langle C_0^{-1}(z_0;1)D\left(\frac{1}{\sqrt{t}}\right) x,D\left(\frac{1}{\sqrt{t}}\right) x\rangle\\&
\le\max_{|y|=1}\langle C_0^{-1}(z_0;1)y,y\rangle \left|D\left(\frac{1}{\sqrt{t}}\right) x\right|^2
=\frac{\left|D\left(\frac{1}{\sqrt{t}}\right) x\right|^2}{\displaystyle\min_{|y|=1}\langle C_0(z_0;1)y,y\rangle}.\end{aligned}
\]
By \eqref{stimaC_0}
\[
\displaystyle\min_{|y|=1}\langle C_{0}(z_{0};1)y,y\rangle\geq\frac{1}{\Lambda
}\min_{|y|=1}\langle{\tilde{C}}(1)y,y\rangle=\frac{\lambda_{\tilde{C}}%
}{\Lambda}%
\]
and the last inequality in \eqref{3.10BCLP2} follows. Analogously the first
one can be proved.
\end{proof}

Collecting the results in Lemma \ref{lemma2} and Lemma \ref{lemma3} we easily
get the following:

\begin{proposition}
Let $\widetilde{C}$ be defined as in \eqref{Ctilde}. There exist positive
constants $T$ and $m$, depending only on the operator $L$, such that the
following inequalities hold for every $t\in\lbrack-2T,2T]$, every $z_{0}%
\in\mathbb{R}^{N+1}$ and every $x\in\mathbb{R}^{N}$:

\begin{itemize}
\item[(a)] $\displaystyle\frac{1}{m}\langle\widetilde{C}(t)x,x\rangle
\leq\langle C(z_{0};t)x,x\rangle\leq m\langle\widetilde{C}(t)x,x\rangle$;

\item[(b)] $\displaystyle\frac{1}{m}\operatorname{det}\widetilde{C}%
(t)\leq\operatorname{det}C(z_{0};t)\leq m\operatorname{det}\widetilde{C}(t)$;

\item[(c)] $\displaystyle\frac{1}{m}\left\vert D\left(  \frac{1}{\sqrt{t}%
}\right)  x\right\vert ^{2}\leq\langle C^{-1}(z_{0};t)x,x\rangle\leq
m\left\vert D\left(  \frac{1}{\sqrt{t}}\right)  x\right\vert ^{2}$.
\end{itemize}

\label{propapp}
\end{proposition}

The above estimates, together with the procedure in \cite[proof of
Proposition 2.7]{DP}, imply the uniform bounds in Theorem
\ref{Thm main appendix} for $\gamma(z_0;z)$. To prove analogous estimates for $\gamma(z_0;z^{-1})$ and its derivatives, one can proceed in a similar way.


\end{document}